\documentclass[5p]{elsarticle}

\usepackage{lipsum}
\makeatletter
\def\ps@pprintTitle{%
 \let\@oddhead\@empty
 \let\@evenhead\@empty
 \def\@oddfoot{}%
 \let\@evenfoot\@oddfoot}
\makeatother

\usepackage[T1]{fontenc}
\usepackage{comment}

\usepackage{amsmath}
\usepackage{amssymb}
\usepackage{amsfonts}
\usepackage{amsthm}
\usepackage{tikz}
\usepackage{pgfplots}
\usepackage{graphicx,pst-all}
\usepackage{subfig}
\usepackage{enumerate}
\usepackage[shortlabels]{enumitem}
\usepackage{caption}

\newcommand{\cC}{\mathcal{C}}
\newcommand{\cB}{\mathcal{B}}

\newcommand{\cA}{\mathcal{A}}

\newcommand{\cL}{\mathcal{L}}

\newcommand{\cD}{\mathcal{D}}

\newcommand{\cX}{\mathcal{X}}
\newcommand{\cY}{\mathcal{Y}}

\usetikzlibrary{arrows}
\usetikzlibrary{fit}\usetikzlibrary{calc}
  \pgfdeclarelayer{background}
  \pgfsetlayers{background,main}

\theoremstyle{theorem}
\newtheorem{Prop}{Proposition}[section]
\newtheorem{Lem}[Prop]{Lemma}
\newtheorem{Thm}[Prop]{Theorem}

\theoremstyle{definition}

\newtheorem{Rem}[Prop]{Remark}

\newcommand{\eps}{\varepsilon}

\newcommand{\N}{{\mathbb{N}}}

\newcommand{\R}{{\mathbb{R}}}
\newcommand{\C}{{\mathbb{C}}}
\newcommand{\K}{{\mathbb{K}}}

\newcommand{\setdef}[2]{\left\{\, #1 \left|\, \vphantom{#1} #2\right.\right\}}

\newcommand{\ds}[1]{{\rm \, d} #1 \,}

\DeclareMathOperator{\RE}{Re}

\DeclareMathOperator{\loc}{loc}

\newenvironment{smallpmatrix}
{\left(\begin{smallmatrix}}
{\end{smallmatrix}\right)}

\newenvironment{smallbmatrix}
{\left[\begin{smallmatrix}}
{\end{smallmatrix}\right]}

\def\tb#1{\textcolor[rgb]{0.00,0.00,0.00}{#1}}

\begin{document}

\begin{frontmatter}

\title{Funnel control for a moving water tank\tnoteref{thanks}}
\tnotetext[thanks]{This work was supported by the German Research Foundation (Deutsche Forschungsgemeinschaft) via the grants BE 6263/1-1 and RE 2917/4-1.}

\author[Paderborn]{Thomas Berger}\ead{thomas.berger@math.upb.de}
\author[Hamburg]{Marc Puche}\ead{marc.puche@uni-hamburg.de}
\author[Hamburg,Twente]{Felix L.~Schwenninger}\ead{f.l.schwenninger@utwente.nl}

\address[Paderborn]{Institut f\"ur Mathematik, Universit\"at Paderborn, Warburger Str.~100, 33098~Paderborn, Germany}
\address[Hamburg]{Fachbereich Mathematik, Universit\"at Hamburg, Bundesstra{\ss}e~55, 20146~Hamburg, Germany}
\address[Twente]{Department of Applied Mathematics, University of Twente, P.O.~Box 217,
7500 AE Enschede, The Netherlands\vspace*{-0.8cm}}


\begin{abstract}
We study tracking control for a nonlinear moving water tank system modeled by the linearized Saint-Venant equations, where the output is given by the position of the tank and the control input is the force acting on it. For a given reference signal, the objective is that the tracking error evolves within a pre-specified performance funnel.
Exploiting recent results in funnel control, this can be achieved by showing that \textit{inter alia} the system's internal dynamics are bounded-input, bounded-output stable.
\end{abstract}

\end{frontmatter}


%
\section{Introduction}\label{Sec:Intr}
%

When a liquid-filled containment is subject to movement, the motion of the fluid may have a significant effect on the dynamics of the overall system and is known as {\it sloshing}. The latter phenomenon can be understood as internal dynamics of the system and it is of great importance in a range of applications such as aeronautics and control of containers and vehicles, and has been studied in engineering for a long time, see e.g.~\cite{CardMati17,FeddDohr97,GrahRodr51,GrunBern99,VenuBern96,YanoYosh96}.

The standard model for the one-dimensional movement of a fluid is given by the Saint-Venant equations, which is a system of nonlinear hyperbolic partial differential equations (PDEs). Models of a moving water tank involving these equations without friction have been studied in various articles, where the control is the acceleration and the output is the position of the tank. The first approach appears in~\cite{DuboPeti99} where a flat output for the linearized model is constructed. Several additional control problems related to this model are studied in~\cite{PetiRouc02} and it is proved that the linearization is steady-state controllable. Moreover, the seminal work~\cite{Coro02} shows that the nonlinear model is locally controllable around any steady state. 
Different stabilization approaches by state and output feedback using Lyapunov functions are studied in~\cite{PrieHall04}. In~\cite{AuroBonn11} observers are designed to estimate the horizontal currents by exploiting the symmetries in the Saint-Venant equations. Convergence of the estimates to the actual states is studied for the linearized model. In~\cite{CardMati17} a port-Hamiltonian formulation of the system is provided as a mixed finite-infinite dimensional system. For a recent numerical treatment of a truck with a fluid basin see e.g.~\cite{GerdKimm15}.

In this note we consider output trajectory tracking for moving water tank systems by funnel control. The concept of funnel control was developed in~\cite{IlchRyan02b}, see also the survey~\cite{IlchRyan08}. The funnel controller is an adaptive controller of high-gain type and proved its potential for tracking problems in various applications, such as temperature control of chemical reactor models~\cite{IlchTren04}, control of industrial servo-systems~\cite{Hack17} and underactuated multibody systems~\cite{BergOtto19}, voltage and current control of electrical circuits~\cite{BergReis14a}, control of peak inspiratory pressure~\cite{PompWeye15} and adaptive cruise control~\cite{BergRaue20}. We like to emphasize that the funnel controller is a model-free feedback controller, i.e., it does not require specific system parameters for feasibility. This makes it a suitable choice for the application to the water tank system, for which we assume that it contains a non-vanishing friction term as modeled in the Saint-Venant equations e.g.\ in~\cite{BastCoro16}, but the exact shape/magnitude of this term is unknown and not available to the controller.

While funnel control is known to work for a large class of functional differential equations with higher relative degree as shown in~\cite{BergLe18a} (cf.\ also Section~\ref{Sec:FunCon}), it is often not clear if a particular system involving internal dynamics governed by PDEs are encompassed by these results. Recently~\cite{BergPuch20a}, we have outlined an abstract framework to answer this question affirmatively. In the present work we follow this approach to show that tracking with prescribed transient behaviour of the moving tank --- subject to sloshing effects modeled via the linearized shallow water equations --- can indeed be achieved by funnel control.

\subsection{Nomenclature}\label{Ssec:Nomencl}

In the following let $\N$ denote the natural numbers, $\N_0 = \N \cup\{0\}$, and $\R_{\ge 0} =[0,\infty)$. We write $\C_{\omega}=\setdef{\lambda\in\C}{\RE \lambda>\omega}$ for $\omega\in\R$ and $\C_+=\C_0$. For a Hilbert space $X$, $L^p(I;X)$ denotes the usual Lebesgue--Bochner space of (strongly) measurable functions $f:I\to X$, $I\subseteq\R$ an interval, where $p\in[1,\infty]$. 
We write $\|\cdot\|_\infty$ for $\|\cdot\|_{L^\infty(\R_{\ge 0};X)}$. By $L^\infty_{\loc}(I;X)$ we denote the set of measurable and locally essentially bounded functions $f:I\to X$, by $W^{k,p}(I;X)$, $k\in\N_0$, the Sobolev space of $k$-times weakly differentiable functions $f:I\to X$ such that $f,\dot f,\ldots, f^{(k)}\in L^p(I;X)$, and by $\cC^k(I;X)$ the set of $k$-times continuously differentiable functions $f:I\to X$, $k\in\N_0\cup\{\infty\}$, where $\cC(I;X):=\cC^0(I;\K^n)$.
 By $\cB(\cX;\cY)$, where $\cX, \cY$ are Hilbert spaces, we denote the set of all bounded linear operators $\cA:\cX\to\cY$. The  symbol ``$\lesssim$'' is a placeholder for ``$\leq c\,\cdot$'' where the multiplicative constant $c$ is independent of the variables occurring in the inequality.

\subsection{The Model}\label{Ssec:MathModel}

In the present paper we study the horizontal movement of a water tank as depicted in Fig.~\ref{Fig:Tank}, where we neglect the wheels' inertia and friction between the wheels and the ground.
\begin{figure}[h!t]
  \centering
\resizebox{0.45\textwidth}{!}{
\begin{tikzpicture}

\draw (0.6,-0.55) circle (15pt); \draw (4.4,-0.55) circle (15pt);
\draw[line width=1.1pt] (0,2) -- (0,0) -- (5,0) -- (5,2);

\draw[->,line width=1.5pt] (-1.5,1) -- (0,1)  node[midway, above]{\large $u(t)$};

\draw[dashed] (0,0) -- (0,-2);
\draw[->,line width=1.5pt] (-2.5,-2) -- (0,-2)  node[midway, above]{\large $y(t)$};
\draw[line width=1.5pt] (-2.5,-1.9) -- (-2.5,-2.1);

\draw[->,line width=1.5pt] (6,1) -- (6,0) node[midway, right]{\large $g$};

\draw (-2.5,-1.1) -- (7,-1.1);

\draw[line width=1.1pt, domain=0:5,smooth,variable=\x,blue] plot ({\x},{1*sqrt(3*\x)*sin(100*\x)*exp(-1*\x)+1});
\fill[color=blue!15,domain=0:5] (0,0)-- plot (\x,{1*sqrt(3*\x)*sin(100*\x)*exp(-1*\x)+1})--(5,0)-- (0,0);

\draw[red,line width=1.1pt,->] (0,0) -- (0,1) node[midway, right]{\large $h(t,\zeta)$};
\draw[red,line width=1.1pt,->] (0,0) -- (5,0) node[midway, below]{\large $\zeta$};
\node at (-0.15,-0.2) {\large \color{red} 0};
\node at (5.15,-0.2) {\large \color{red} 1};

\draw[line width=1.2pt] (2,0) -- (2,0.9);
\draw[->,line width=1.2pt] (2,0.45) -- (3,0.45) node[pos=1, right]{\large $v(t,\zeta)$};
\draw[line width=1.1pt] (5,0) -- (5,2);
\end{tikzpicture}
}
\caption{Horizontal movement of a water tank.}
\label{Fig:Tank}
\end{figure}
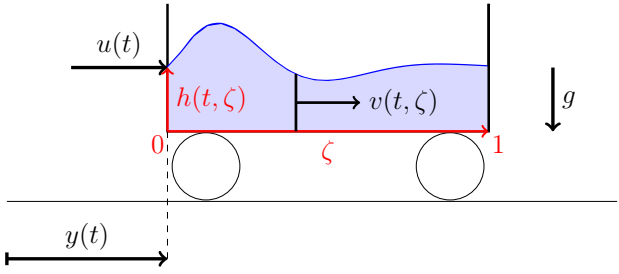
We assume that there is an external force acting on the water tank, which we denote by $u(t)$ as this will be the control input of the resulting system, cf.\ also Section~\ref{Ssec:ContrObj}. The measurement output is the horizontal position~$y(t)$ of the water tank, and the mass of the empty tank is denoted by $m$. The dynamics of the water under gravity~$g$ are described by the {\it Saint-Venant equations} (first derived in~\cite{Sain71}; also called one-dimensional {\it shallow water equations})
\begin{equation}\label{eq:SVeq}
\begin{aligned}
  \partial_t h + \partial_\zeta(h v)&= 0,\\
  \partial_t v +  \partial_\zeta \left( \frac{v^2}{2} + gh\right)+hS\left(\frac{v}{h}\right)&= - \ddot y
\end{aligned}
\end{equation}
with  boundary conditions $v(t,0) = v(t,1) = 0$. Here $h:\R_{\ge 0}\times [0,1]\to \R$ denotes the height profile and $v:\R_{\ge 0}\times [0,1]\to \R$ the (relative) horizontal velocity profile, where the length of the container is normalized to~$1$.
The friction term~$S:\R\rightarrow\R$ is typically modeled as the sum of a high velocity coefficient of the form $C_Sv^2/h^2$ and a viscous drag of the form $C_Dv/h$ for some positive constants $C_S,C_D$. In the present paper, we do not specify the function~$S$, but we do assume that $S(0)=0$ and $S'(0)>0$. The condition $S(0)=0$ means that, whenever the velocity is zero, then there is no friction. The condition $S'(0)>0$ means that the viscous drag does not vanish; this is the case in most real-world non-ideal situations, but sometimes neglected in the literature,
 see e.g.~\cite[Sec.~1.4]{BastCoro16}.

For a derivation of the Saint-Venant equations~\eqref{eq:SVeq} of a moving water tank we refer to~\cite{CardMati17,PetiRouc02}, see also the references therein. The friction term in the model is the general version of that used in~\cite[Sec.~1.4]{BastCoro16}. Let us emphasize that in our framework the input is the force acting on the water tank, which can be manipulated using an engine for instance. In contrast to this, in~\cite{Coro02, PetiRouc02} the acceleration of the tank is used as input, but this can usually not be influenced directly. Note that --- in the presence of sloshing ---  the applied force does not equal the product of the tank's mass and acceleration.
We also stress that, if the acceleration is used as input, then the input-output relation is given by the simple double integrator $\ddot y = u$, and the Saint-Venant equations~\eqref{eq:SVeq} do not affect this relation. 

As shown in~\cite{DuboPeti99, PetiRouc02}, the linearization of the Saint-Venant equations is relevant in the context of control since it provides a model which is much simpler to solve (both analytically and numerically) and still is an insightful approximation for motion planning purposes.
Therefore, we restrict ourselves to the linearization of~\eqref{eq:SVeq} around the steady state $(h_{0},0)=(\int_{0}^{1}h(0,\zeta)\mathrm{d}\zeta,0)$, given by

\begin{equation}\label{eq:SVlin}
    \partial_t z = Az +b \ddot y = -\begin{bmatrix}0&h_0 \partial_\zeta\\g\partial_\zeta &2\mu\end{bmatrix}z+\begin{pmatrix}0\\-1\end{pmatrix}\ddot{y}
\end{equation}
with boundary conditions $z_2(t,0) = z_2(t,1) = 0$,
$\mu = \tfrac12 S'(0) > 0$ and $b=(0,-1)^\top$.
The state space in which $z(t)$ evolves is
     $X=L^2([0,1];\R^2)$
and $A: \cD(A) \subseteq  X \to X$,
\begin{equation}\label{eq:domA}
    \cD(A) = \setdef{(z_1,z_2)\in X}{ \!\!\begin{array}{l} z_1,z_2\in W^{1,2}([0,1];\R),\\ z_2(0) = z_2(1) = 0\end{array}\!\!\!}.
\end{equation}
By conservation of mass in  \eqref{eq:SVlin}, $\int_{0}^{1}z_{1}(t,\zeta)\mathrm{d}\zeta=h_{0}$ for all $t\ge0$.
The model is completed by the momentum
\begin{equation}\label{eq:momentum}
    p(t) := m \dot y(t) + \int_0^1 z_1(t,\zeta) \big(z_2(t,\zeta) + \dot y(t)\big) \ds{\zeta},\ t\ge0.
\end{equation}
Substituting the absolute velocity $x_2=z_{2}+\dot{y}$ for $z_2$, $x_{1}=z_{1}$ and using the balance law $\dot p(t) = u(t)$ and~\eqref{eq:SVlin} we obtain
\begin{align*}
m \ddot y(t) &=   \frac{g}{2} x_1(t,\cdot)^2|_{0}^{1} +2\mu\langle x_{1}(t),x_{2}(t)\rangle-2\mu h_0\dot{y}(t)+u(t),
  \end{align*}
where $\langle f,g\rangle=\int_{0}^{1}f(s)g(s)\mathrm{d}s$.
Altogether, the nonlinear model on the state space $X$ reads
\begin{subequations}\label{eq:InpOut}
\begin{align}\label{eq:InpOutlin}
	  \partial_t x &= A(x+b\dot{y}) \\
 \!  m \ddot y(t) &= \tfrac{g}{2}  x_1(t,\cdot)^{2}|_{0}^{1}\!+\!{2\mu}\langle x_{1}(t),x_{2}(t)\rangle-2\mu h_0\dot{y}(t)+{u(t)} \label{eq:InpOutlin2}
\end{align}
\end{subequations}
with input $u$, state $x$ and output $y$.

\tb{We like to note that system~\eqref{eq:InpOut} is basically a hyperbolic PDE coupled with an ODE (when~\eqref{eq:InpOutlin2} is rewritten as a system of first order equations). Therefore, it might be amenable to stabilization by backstepping methods, which have been successfully used for such systems in the recent past, see e.g.~\cite{DeutGabr21,DiMeBrib18,WangKrst18}. However, we like to emphasize that~\eqref{eq:InpOutlin2} is nonlinear, which is out of the scope of these works, and the funnel control techniques studied in the present work (which do not aim at stabilization) can be directly applied to the present form, which is more natural from a modelling point of view. Still the question how funnel control is related to this broad scope of existing results is not entirely clear and remains interesting.}

\subsection{Control objective --- funnel control}\label{Ssec:ContrObj}

Our goal is to design an output error feedback of the form
\begin{equation} \label{eq:F}
u(t) = F\big(t, e(t), \dot e(t)\big),
\end{equation}
where $e(t) = y(t) - y_{\rm ref}(t)$ is the tracking error and $y_{\rm ref}\in W^{2,\infty}(\R_{\ge 0};\R)$ is a given reference position, which applied to~\eqref{eq:InpOut} results in a closed-loop system that satisfies:
\begin{itemize}
\item the pair $(t,e(t))$ evolves within the prescribed set
    \begin{equation*}
    \mathcal{F}_{\varphi} := \setdef{(t,e)\in\R_{\ge 0} \times\R}{\varphi(t) |e| < 1},\label{eq:perf_funnel}
    \end{equation*}
    which is determined by a function~$\varphi$ belonging to
    \[
     \Phi \!:=\!
    \setdef{
    \varphi\in  \cC^1(\R_{\ge 0};\R)
    }{\!\!\!
    \begin{array}{l}
    \text{ $\varphi, \dot \varphi$ are bounded,}\\
    \text{ $\varphi (\tau)>0$ for all $\tau>0$,}\\
     \text{ } \liminf_{\tau\rightarrow \infty} \varphi(\tau) > 0
    \end{array}
    \!\!\!}
    \]
    and
\item  the signals $u, e, \dot e$ are uniformly bounded on $\R_{\ge0}$.
\end{itemize}

The set $\mathcal{F}_{\varphi}$ is called the \textit{performance funnel}.
Its boundary, the \textit{funnel boundary}, is given by the reciprocal of $\varphi$, see Fig.~\ref{Fig:funnel}. The case $\varphi(0)=0$ is explicitly allowed and puts no restriction on the initial value since $\varphi(0) |e(0)| < 1$; in this case the funnel boundary $1/\varphi$ has a pole at $t=0$.

 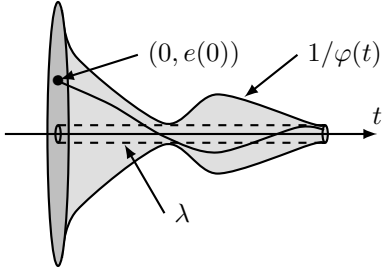
\begin{figure}[h]
  \begin{center}
\begin{tikzpicture}[scale=0.35]
\tikzset{>=latex}
  \filldraw[color=gray!25] plot[smooth] coordinates {(0.15,4.7)(0.7,2.9)(4,0.4)(6,1.5)(9.5,0.4)(10,0.333)(10.01,0.331)(10.041,0.3) (10.041,-0.3)(10.01,-0.331)(10,-0.333)(9.5,-0.4)(6,-1.5)(4,-0.4)(0.7,-2.9)(0.15,-4.7)};
  \draw[thick] plot[smooth] coordinates {(0.15,4.7)(0.7,2.9)(4,0.4)(6,1.5)(9.5,0.4)(10,0.333)(10.01,0.331)(10.041,0.3)};
  \draw[thick] plot[smooth] coordinates {(10.041,-0.3)(10.01,-0.331)(10,-0.333)(9.5,-0.4)(6,-1.5)(4,-0.4)(0.7,-2.9)(0.15,-4.7)};
  \draw[thick,fill=lightgray] (0,0) ellipse (0.4 and 5);
  \draw[thick] (0,0) ellipse (0.1 and 0.333);
  \draw[thick,fill=gray!25] (10.041,0) ellipse (0.1 and 0.333);
  \draw[thick] plot[smooth] coordinates {(0,2)(2,1.1)(4,-0.1)(6,-0.7)(9,0.25)(10,0.15)};
  \draw[thick,->] (-2,0)--(12,0) node[right,above]{\normalsize$t$};
  \draw[thick,dashed](0,0.333)--(10,0.333);
  \draw[thick,dashed](0,-0.333)--(10,-0.333);
  \node [black] at (0,2) {\textbullet};
  \draw[->,thick](4,-3)node[right]{\normalsize$\lambda$}--(2.5,-0.4);
  \draw[->,thick](3,3)node[right]{\normalsize$(0,e(0))$}--(0.07,2.07);
  \draw[->,thick](9,3)node[right]{\normalsize$1/\varphi(t)$}--(7,1.4);
\end{tikzpicture}
\end{center}
 \vspace*{-2mm}
 \caption{Error evolution in a funnel $\mathcal F_{\varphi}$ with boundary $1/\varphi(t)$.}
 \label{Fig:funnel}
 \end{figure}
On the other hand, note that boundedness of $\varphi$ implies that there exists $\lambda>0$ such that $1/\varphi(t)\geq\lambda$ for all $t > 0$. This implies that signals evolving in $\mathcal{F}_{\varphi}$ are not forced to converge to $0$ asymptotically. Furthermore, the funnel boundary is not necessarily monotonically decreasing
and there are situations, like in the presence of periodic disturbances, where widening the funnel over some later time interval might be beneficial.
It was shown in~\cite{BergLe18a} that for $\varphi_0,\varphi_1\in\Phi$, the following choice for $F$ in \eqref{eq:F}
\begin{equation}\label{eq:fun-con}
\begin{aligned}
F(t,e(t),\dot{e}(t)) &= - k_1(t) \big( \dot e(t) + k_0(t) e(t)\big),\\
k_0(t) &= \frac{1}{1-\varphi_0(t)^2\|e(t)\|^2},\\
k_1(t) &= \frac{1}{1-\varphi_1(t)^2\|\dot e(t) + k_0(t) e(t)\|^2},
\end{aligned}
\end{equation}
achieves the above control objective for a large class of nonlinear systems with relative degree two. In the present paper we extend this result and show feasibility of~\eqref{eq:fun-con} for the model described by~\eqref{eq:InpOut}. We highlight that the functions $\varphi_0,\varphi_1$ are design parameters in the control law~\eqref{eq:fun-con}.
 Typically, the specific application dictates the constraints on the tracking error and thus indicates suitable choices.

In~\cite{BergPuch20a} --- extending the findings from~\cite{BergLe18a} --- it was shown that the controller~\eqref{eq:fun-con} is feasible for nonlinear systems of the form
\begin{equation}\label{eq:nonlSys}\tag{Sys}
\begin{aligned}
\ddot y(t)&=\mathcal{S}(y,\dot{y})(t) + \gamma\, u(t)\\
\big(y(0),\dot y(0)\big) &= \big(y^0,y^1\big) \in \R^2,
\end{aligned}
\end{equation}
where, under structural assumptions, the operator $\mathcal{S}$ may in particular incorporate input-output dynamics from an infinite-dimensional well-posed linear system.
We note that corresponding results hold for systems with relative degree other than two, but this special case is sufficient for the present article. General sufficient conditions on the operator $\mathcal{S}$ guaranteeing feasibility of the controller~\eqref{eq:fun-con} were given in~\cite{BergLe18a,HackHopf13,IlchSeli16} and~\cite{IlchRyan02b} before, while suitable adaptions allowing for truly infinite-dimensional internal dynamics were finally explored in~\cite{BergPuch20a}. For details on the structural assumptions on the systems class and the operator $\mathcal{S}$ and the relation to prior results we refer to~\cite{BergPuch20a}.

\subsection{Organization of the present paper}\label{Ssec:Orga}

In Section~\ref{Sec:FunCon} we formulate the main result of this article, stating  that the funnel control objective for the model of the moving water tank is achieved in the sense of Section~\ref{Ssec:ContrObj}.
For this, it suffices to verify the conditions identified in~\cite{BergPuch20a},
which is done in Section~\ref{Sec:LinMod} by considering the model
in the framework of well-posed linear systems. \tb{Possible extensions of the results to the case of steady states corresponding to non-zero control values and invoking space-dependent friction terms are discussed in Section~\ref{Sec:Ext}.} The application of the controller to the moving water tank system is illustrated by a simulation in Section~\ref{Sec:Sim}.

\section{Main result}\label{Sec:FunCon}
%

In this section we formulate how the funnel controller~\eqref{eq:fun-con} described in Subsection~\ref{Ssec:ContrObj}  achieves the control objective for system~\eqref{eq:InpOut}  --- this is the main result of the article. 
The initial conditions for~\eqref{eq:InpOut} are
\begin{equation}\label{eq:IClin}
\begin{aligned}
    x(0) = x_0 \in X, \quad
     \big(y(0),\dot y(0)\big) = \big(y^0,y^1\big)\in\R^2.
\end{aligned}
\end{equation}

We call $(x,y):[0,\omega)\rightarrow {X}\times\R$ a strong solution of \eqref{eq:InpOut}--\eqref{eq:IClin} on an interval $[0,\omega)$, if \footnote{For the definition of $X_{-1}$ see Sec.~\ref{Sec:LinMod}.}
\begin{itemize}
\item $y\in W_{\loc}^{2,1}([0,\omega);\R)$ and $x\in \cC([0,\omega);X)$,
\item the initial conditions \eqref{eq:IClin} hold,
\item  $x\in W_{\mathrm{loc}}^{1,1}([0,\omega);X_{-1})$
and~\eqref{eq:InpOutlin} holds for a.e.\ $t\in[0,\omega)$ as equation in $X_{-1}$,
\item $y$ satisfies \eqref{eq:InpOutlin2} for a.a.\ $t\in[0,\omega)$. 
\end{itemize}
In other words, $x$ is a strong solution of~\eqref{eq:InpOutlin} and $y$ is a Carath\'eodory solution of~\eqref{eq:InpOutlin2}.
A solution $(x,y)$ is called {\em classical}, if $x\in \cC^1([0,\omega);X)$ and $y\in \cC^2([0,\omega);\R^2)$; it is called {\em global}, if it can be extended to $\R_{\ge 0}$.

\begin{Thm}\label{Thm:funcon-PDE}
   Let $y_{\rm ref}\in W^{2,\infty}(\R_{\ge 0};\R)$, $\varphi_0,\varphi_{1}\in\Phi$ and $(y^{0},y^{1})\in \R^{2}$, $x_0\in X$ and $v_0\in\R$
 be such that
\begin{equation}\label{eq:initialerror}\begin{array}{ll}
&x_0 + b v_0\in \cD(A),\\
&\varphi_0(0) |y^0-y_{\rm ref}(0)|<1, \quad\text{and }\\
&\varphi_1(0) | y^1-\dot y_{\rm ref}(0) + k_0(0) \big(y^0-y_{\rm ref}(0)\big) | < 1.
\end{array}
\end{equation}
Then the closed-loop system~\eqref{eq:InpOut}--\eqref{eq:IClin}
has a unique global strong solution $(x,y):\R_{\ge0}\rightarrow {X}\times\R$. Moreover, the following properties hold.
\begin{enumerate}
\item \label{thmi1}The functions $k_{0},k_{1},u$ and $x,y,\dot{y}$ are bounded.
\item  \label{thmi2}The error~$e = y-y_{\rm ref}$ is uniformly bounded away from the funnel boundary in the following sense:
\begin{equation}\label{eq:bound-error}
 \exists\, \varepsilon>0\  \forall\, t>0:\ |e(t)|\le \varphi_0(t)^{-1} - \varepsilon.
\end{equation}
\item \label{thmi3} If $y^1=v_0$, then the solution is a classical solution.
\end{enumerate}
\end{Thm}

\begin{proof}
{\it Step 1:}
We rewrite \eqref{eq:InpOut},~\eqref{eq:IClin} in the form of equation~\eqref{eq:nonlSys}, obtaining
\begin{align}%
\ddot{y}(t)&=\mathcal{T}(\dot y)(t)+\frac{u(t)}{m},\label{eq:Teq1}
\end{align}
where the mapping $\mathcal{T}$ is formally given by
\begin{align*}
\mathcal{T}(\eta)(t)={}&\frac{g}{2m}  x_1(t,\cdot)^{2}|_{0}^{1}  + \frac{2\mu}{m} \big( \langle x_{1}(t),x_{2}(t)\rangle - h_0\eta(t)\big)\notag \\
={}&F\big(\widetilde{\mathcal{T}}(\eta)(t),Cx(t)\big)
\end{align*}
with $x$ being the strong solution of
\begin{equation}
  \dot{x}(t)= A\big(x(t)+ b\eta(t)\big), \quad x(0)=x_0,\label{eq:Teq0}
\end{equation}
where $A,b$ are defined in \eqref{eq:SVlin}--\eqref{eq:domA} and we use the notation
\begin{equation}\label{eq:opFTC}
\begin{aligned}
	&F:\R\times\R^{2}\to\R,(\alpha,\beta)\mapsto \frac{g}{2m}\beta_{1}\beta_{2}+\frac{2\mu}{m}\alpha,\\
	&\widetilde{\mathcal{T}}:\cC(\R_{\ge0};\R)\to L_{\mathrm{loc}}^{\infty},\ \eta\mapsto\langle x_{1},x_{2}\rangle -h_{0}\eta,\\
	&C:\cD(C)\subset X\to \R^{2}, x\mapsto (C_{1}x,C_{2}x)^\top,\\
	&C_{i}x:=(x_{1}(1)+(-1)^{i}x_{1}(0)), \quad i=\{1,2\}.
\end{aligned}
\end{equation}
Since the operator $C$ acts as point evaluation of functions in space, the domain $\cD(C)\subset X$ has to be chosen suitably, see~\eqref{eq:domC}.
Also note that $\mathcal{T}$ depends on $x=x(t,\zeta)$ which in turn is given through $\eta$ and $x_{0}$ as the solution of~\eqref{eq:Teq0}, the existence of which is an outcome of Step~3 below.
\smallskip

{\it Step 2:} We show that $\mathcal{T}$ is well-defined from $\cC(\R_{\ge0};\R)$ to $L^\infty_{\loc}(\R_{\ge 0};\R)$ and, in particular, that the mapping
\begin{align*}
\mathcal{F}:\cC(\R_{\ge0};\R)\to L_{\mathrm{loc}}^{\infty}(\R_{\ge0};X\times \R^{2}),\ \eta\mapsto \begin{smallbmatrix}I\\C\end{smallbmatrix}x,
\end{align*}
associated with the PDE~\eqref{eq:Teq0}  is  well-defined. Moreover we show that
\begin{equation}
\label{eq:estimate:xCx}
\max\{\|x\|_{\infty},\|Cx\|_{\infty}\}\lesssim c_{x_{0}}+\|\eta\|_{\infty}
\end{equation}
 for all $\eta\in \cC(\R_{\ge0};\R)\cap L^{\infty}(\R_{\ge0};\R)$, where $c_{x_{0}}=0$ if $x_{0}=0$.
This step is performed in Proposition \ref{prop:BIBO} by showing that the triple $(A,Ab,\begin{smallbmatrix}I\\C\end{smallbmatrix})$ defines a well-posed bounded-input, bounded-output stable linear system.
\smallskip

{\it Step 3:} Note that $\widetilde{\mathcal{T}}$ is uniformly Lipschitz on bounded sets, i.e., for any $R>0$ there exists $L=L_{x_{0}}>0$ such that
\begin{equation}\label{eq:Lipschitz}
\|\widetilde{\mathcal{T}}(\eta)-\widetilde{\mathcal{T}}(\hat{\eta})\|_{L^{\infty}([0,t];\R)}\leq L\|\eta-\hat{\eta}\|_{L^{\infty}([0,t];\R)}
\end{equation}
for all $t>0$ and $\eta,\hat{\eta}\in \setdef{\eta\in\cC([0,t];\R)}{ \|\eta\|_{\infty}\leq R}$. This follows easily from \eqref{eq:estimate:xCx} and the fact that an inner product restricted to bounded subsets is uniformly Lipschitz. Furthermore, it is clear that $\widetilde{\mathcal{T}}$ is causal, and, by \eqref{eq:estimate:xCx}, that for all $\eta\in \cC(\R_{\ge0};\R)\cap L^{\infty}(\R_{\ge0};\R)$ we have
\[
    \|\widetilde{\mathcal{T}}(\eta)\|_{\infty}\lesssim \tilde{c}_{x_{0}}+\|\eta\|_{\infty}^{2}+\|\eta\|_{\infty}.
\]
Since $F\in \cC^{1}(\R\times\R^{2};\R)$, we conclude from the above that $\mathcal{T}$ satisfies all the Properties (P1)--(P4) from \cite[Def.~3.1]{BergPuch20a} and thus~\cite[Thms.~2.1~\&~3.3]{BergPuch20a} imply the existence of a global strong solution $(x,y)$ and, together with~\eqref{eq:estimate:xCx}, that Assertions~\ref{thmi1}--\ref{thmi2} hold. Note that mild solutions of~\eqref{eq:InpOutlin} as considered in~\cite{BergPuch20a} are strong solutions, \cite[Thm.~3.8.2]{Staf05}.
\smallskip

{\it Step 4:} We show uniqueness of the solution. First recall that, as a consequence of Proposition~\ref{prop:sec3}, the unique strong solution of~\eqref{eq:Teq0} is given by
  \begin{equation}\label{eq:strongsol}
  x(t)=T(t)x_{0}+\int_{0}^{t} T_{-1}(t-s)B\eta(s)\ds{s},\quad t\ge 0,
  \end{equation}
where $B=A_{-1} b$. Thus, invoking $\eta = \dot y$, it suffices to show uniqueness of the solution~$y$ of~\eqref{eq:Teq1}. Assume that $\bar y\in W^{2,1}_{\loc}(\R_{\ge 0};\R)$ is another solution of~\eqref{eq:Teq1} with $\bar y(0) = y^0$ and $\dot{\bar y}(0) = y^1$, and let $\bar x$ be the strong solution of~\eqref{eq:InpOutlin} with $y=\bar y$ and $\bar x(0) = x_0$. Define $t_0 := \inf\setdef{t\ge 0}{ \dot y(t) \neq \dot{\bar y}(t)}$ and assume, seeking a contradiction, $t_0<\infty$. Clearly, $x(t) = \bar x(t)$ for all $t\in[0,t_0]$, hence, by Proposition~\ref{prop:BIBO}, there exists $c_0>0$ such that for all $t\in[t_0,t_0+1]$ we have $\|Cx(t)-C\bar x(t)\|_{\R^2} \le c_0 \sup_{s\in [t_0,t_0+t]} |\dot y(s) - \dot{\bar y}(s)|$. Furthermore, for $R:= \max\{\|\dot y\|_\infty, \|\dot{\bar y}\|_\infty\}$ we obtain a constant $L>0$ such that~\eqref{eq:Lipschitz} holds. The preparations are completed by observing that, invoking Step~3 and~\eqref{eq:estimate:xCx}, there exist compact subsets $K_1\subseteq\R$ and $K_2 \subseteq\R^2$ such that $\widetilde{\mathcal{T}}(\dot y)(s), \widetilde{\mathcal{T}}(\dot{\bar y})(s)\in K_1$ and $(Cx)(s), (C\bar x)(s)\in K_2$ for all $s\in [t_0,t_0+1]$. Since $F$ from~\eqref{eq:opFTC} is in $\cC^{1}(\R\times\R^{2};\R)$,
 there exists $c_1>0$ such that
\[
    |F(\alpha_1,\beta_1) - F(\alpha_2,\beta_2)| \le c_1 (|\alpha_1-\alpha_2| + \|\beta_1-\beta_2\|_{\R^2})
\]
for all $\alpha_1,\alpha_2\in K_1$ and $\beta_1,\beta_2\in K_2$. Now choose $\sigma\in (0,1)$ such that $\sigma c_1 (L+c_0) < 1$. By definition of $t_0$,  $\dot{y}(t_{0})=\dot{\bar{y}}(t_{0})$ and there exists $t\in [t_0, t_0+\sigma]$ such that
\[
   |\dot y(t) - \dot{\bar y}(t)| = \sup_{s\in [t_0,t_0+\sigma]} |\dot y(s) - \dot{\bar y}(s)| =: \eps > 0.
\]
Hence, integrating~\eqref{eq:Teq1} yields the contradiction
\begin{align*}
    \eps &= |\dot y(t) - \dot{\bar y}(t)|\\
  &\le \int_{t_0}^t \big|F\big(\widetilde{\mathcal{T}}(\dot y)(s),Cx(s)\big) - F\big(\widetilde{\mathcal{T}}(\dot{\bar y})(s),C\bar x(s)\big)\big| {\rm d}s \\
  &\le \sigma c_1 (L+c_0) \sup_{s\in [t_0,t_0+t]} |\dot y(s) - \dot{\bar y}(s)| < \eps.
\end{align*}
Thus, $\dot y = \dot{\bar y}$ and by $y(0) = y^0 = \bar y(0)$, it follows that $y = \bar y$.
\smallskip

{\it Step 5:} We show~\ref{thmi3}. Note that $\dot{y}\in W_{\mathrm{loc}}^{1,2}(\R_{\ge0};\R)$ by Step~3 and $\dot y(0) = v_0$ by assumption. Invoking well-posedness together with $x(0) + b \dot y(0)\in\cD(A)$ and~\cite[Prop.~4.6]{TucsWeis14}, the solution $x$ of~\eqref{eq:Teq0} is indeed a classical solution and $\mathcal{F}(\eta)\in \cC(\R_{\ge0};X\times\R^{2})$. This implies that $\mathcal{T}(\dot{y})(t)$ is continuous, whence $\ddot{y}$ is continuous by~\eqref{eq:Teq1}, which proves the claim.
\end{proof}

\vspace{-5mm}

\tb{
\begin{Rem}
We like to emphasize that the funnel controller~\eqref{eq:fun-con} does not require any knowledge of system parameters or initial values. Therefore, it is robust with respect to (arbitrary) uncertainties in these parameters. More precisely, for any fixed controller parameters $\varphi_0,\varphi_{1}\in\Phi$ the controller~\eqref{eq:fun-con} is feasible in the sense of Theorem~\ref{Thm:funcon-PDE} for any system parameters $\mu>0,\ m>0,\ h_0, g\in\R$, any reference signal $y_{\rm ref}\in W^{2,\infty}(\R_{\ge 0};\R)$ and any initial values $(y^{0},y^{1})\in \R^{2}$, $x_0\in X$ which satisfy~\eqref{eq:initialerror}. In particular, for any such parameters the controller achieves the prescribed performance of the tracking error as in~\eqref{eq:bound-error}, without the need to modify or tune the controller. Via the gain functions~$k_0$ and~$k_1$ in~\eqref{eq:fun-con} the controller is able to adapt its behavior to the specific situation.
\end{Rem}
}

\vspace{-3mm}

%
\section{Linearized model -- abstract framework}\label{Sec:LinMod}
%

In this section we collect and derive the results required for Step 2 in the proof of Theorem~\ref{Thm:funcon-PDE} by using the framework of well-posed linear systems and showing bounded-input, bounded-output stability of the considered systems. 


\tb{Let us recall a few basics from semigroup theory and admissible operators in the context of linear systems, which can all be found e.g.\ in \cite{TucsWeis09}. A {\it semigroup}  $({T}(t))_{t\ge0}$ on $X$ is a $\mathcal{B}(X;X)$-valued map satisfying ${T}(0)=I_{X}$ and ${T}(t+s)={T}(t){T}(s)$, $s,t\geq0$, where $I_{X}$ denotes the identity operator. Furthermore, we assume that semigroups are strongly continuous, i.e., $t\mapsto {T}(t)x$ is continuous for every $x\in X$. Semigroups are characterized by their generator~$\cA$, which is a possibly unbounded operator on~$X$. The growth bound of the semigroup is the infimum over all $\omega\in \R$ such that $\sup_{t\ge0}\|\mathrm{e}^{-t\omega}T(t)\|<\infty$.}
 For $\beta\in\C$ in the resolvent set $\rho(A)$ of the generator $A$, we denote by $X_{-1}$ the completion of~$X$ with respect to the norm $\|\cdot\|_{X_{-1}}=\|(\beta I-A)^{-1}\cdot \|_X$. Recall that $X_{-1}$ is independent of the choice of~$\beta$ and that $(\beta I-{A})$ uniquely extends to a surjective isometry $(\beta I - A_{-1})\in \cB(X;X_{-1})$.
The  semigroup $(T(t))_{t\ge 0}$ has a unique extension to a semigroup $(T_{-1}(t))_{t\ge 0}$ in~$X_{-1}$, which is generated by $A_{-1}$. Furthermore, let $X_1$ be the space $\cD(A)$ equipped with the graph norm of $A$.
For Hilbert spaces $U,Y$, a triple $(A,B,C)\in \cB(X_{1};X)\times \cB(U;X_{-1})\times \cB(X_{1};Y)$ is called a \emph{regular well-posed} system, if for some (hence all) $t>0$,
\begin{enumerate}[label=(\alph*)]
\item\label{prop1} $A$ is the generator of a semigroup $({T}(s))_{s\ge0}$ on $X$;
\item\label{prop2} $\int_{0}^{t}T_{-1}(t-s)Bu(s)\ds{s}\in X$ for all $u\in L^{2}([0,t];U)$;
\item\label{prop3} $(\cD(A),\|\cdot\|_X)\!\to\!L^{2}([0,t];Y),\, x\!\mapsto\! CT(\cdot)x$ is bounded;
\item\label{prop4} there exists a bounded function $\mathbf{G}:\C_{\omega}\to \cB({U};Y)$, $\omega>\omega_{{A}}$, such that for all $r,s\in\C_{\omega}$,
\begin{equation}\label{eq:transfer1}
\mathbf{G}(r)-\mathbf{G}(s)=C((rI-{A})^{-1}-(sI-{A})^{-1})B,
\end{equation}
and $\lim_{s\to\infty,s\in \R}\mathbf{G}(s)v$ exists for every $v\in U$.
\end{enumerate}
Operators satisfying \ref{prop2} and \ref{prop3} are called {\it admissible} in the literature and naturally appear in the theory of  boundary control systems, cf.~\cite{Staf05,TucsWeis09}.
\tb{The function $\mathbf{G}$ is called a {\it transfer function} of $(A,B,C)$ and is uniquely determined up to a constant}.
{

From now on we will, without loss of generality, consider the complexification $L^2([0,1];\C^2)$ of the state space $X$
 and the linear operator $A$ defined in \eqref{eq:SVlin}--\eqref{eq:domA}.
The following is a simple exercise in the context of well-posed systems. We include a short proof for completeness.

\begin{Prop}\label{prop:sec3}
 Let $A$ and $\mu$ be defined as in \eqref{eq:SVlin}--\eqref{eq:domA}, $B=A_{-1}b \in \cB(\R;X_{-1})$ with $b=(0,-1)^\top$ and $C\in \cB(X_1;\C^{2})$ as defined in~\eqref{eq:opFTC}. Then
$A$ generates a contraction semigroup $(T(t))_{t\ge0}$ and the triple
 $(A,B,\begin{smallbmatrix}I\\C\end{smallbmatrix})$ is a regular well-posed system \tb{with} transfer function \[\mathbf{G}:\C_{+}\to \cB(\C;X\times \C\times\C), \mathbf{G}=\left(\mathbf{L}-b,\mathbf{H}+2\sqrt{\tfrac{h_{0}}{g}},0\right)\] where
\begin{align}
\mathbf{L}(\lambda)={}&\tb{\lambda(\lambda I-A)^{-1}b},\\
\mathbf{H}(\lambda)={}& -\sqrt{\tfrac{4h_0\lambda}{g(\lambda+2\mu)}}\,\tanh\left(\tfrac{\sqrt{\lambda(\lambda+2\mu)}}{2\sqrt{h_{0}g}}\right), \ \ \lambda\in\C_{+}.\label{eq:transferfct}
\end{align}
\tb{The restricted semigroup on $X_{\rm exp}=(\ker A)^{\perp}=[(\begin{smallmatrix}1\\0\end{smallmatrix})]^{\perp}$ is well-defined, $\|T(t)\|_{\cB(X_{\rm exp})}\leq\mathrm{e}^{t(-\mu+\mathrm{Re}\sqrt{\mu^{2}-\pi^{2}gh_{0}})}$ for all $t\ge0$ and $B\in \cB(\R;(X_{\rm exp})_{-1})$.}
\end{Prop}
\begin{proof} By standard arguments (e.g.\ a Fourier ansatz or Lumer--Phillips theorem), $A$ generates a contraction semigroup $(T(t))_{t\ge0}$ which even extends to a group, whence Property~\ref{prop1}.
\tb{Moreover, since $A$ has a compact resolvent, it follows that there exists an orthonormal basis of eigenvectors of $A$ with eigenvalues $\theta_{n}^{\pm}=-\mu\pm \mathrm{i}\sqrt{gh_{0}\pi^{2}n^2-\mu^{2}}$, $n\in\N$.
This shows that the semigroup leaves $\ker A=[(\begin{smallmatrix}1\\0\end{smallmatrix})]$ and its orthogonal complement $\ker A^{\perp}$ invariant and that $(T(t)|_{(\ker A)^{\perp}})_{t\ge0}$ has growth bound $\omega=-\mu+\mathrm{Re}\sqrt{\mu^{2}-\pi^{2}gh_{0}}<0$. }
Consider the holomorphic function $\lambda\mapsto(\lambda I-A)^{-1}B$ from $\C_{+}$ to $X$.  
By the resolvent identity, we have that
\[
 (\lambda I-A)^{-1}B=\lambda(\lambda I-A)^{-1}b-b=\lambda z-b,
\]
where $z$ can be \tb{computed by solving the ODE $\lambda z-Az=b$,
\begin{align*}
z=\frac{1}{g\theta}\left(\tfrac{\cosh\left({\theta}{}\right)-1}{\sinh\left({\theta}{}\right)}\begin{smallbmatrix}
\cosh\left({\theta\zeta}{}\right)\\
-\frac{\lambda }{h_{0}\theta}\sinh\left({\theta\zeta}{}\right)
\end{smallbmatrix}+\begin{smallbmatrix}
-\sinh\left({\theta\zeta}{}\right)\\
\frac{\lambda }{h_{0}\theta}\left(\cosh\left({\theta\zeta}\right)-1\right)
\end{smallbmatrix}\right),
\end{align*}
with $\theta=\frac{1}{\sqrt{h_{0}g}}\sqrt{\lambda(\lambda+2\mu)}$.}
Since $(\lambda I-A)^{-1}B=\mathbf{L}(\lambda)-b$ is bounded in $\lambda$ on the half-plane $\mathbb{C}_{\mu+1}$, it follows that $B$ is admissible, Property \ref{prop2}, by \cite[Thm.~5.2.2]{TucsWeis09}. Thus, $(A,B,I)$ is well-posed. Similarly, one can show that Property \ref{prop3} holds for $C$ using \cite[Cor.~5.2.4]{TucsWeis09}.
Using the \tb{explicit} formula for \tb{$\mathbf{L}(\lambda) = \lambda z$} shows that $\mathbf{H}$ \tb{has the form as in~\eqref{eq:transferfct} and} is indeed a transfer function for $(A,B,C)$ and \tb{$\lim_{s\to\infty,s\in\R}\mathbf{G}(s)=0$}. Thus Property~\ref{prop4} holds. \tb{Clearly, $\ker A=[(\begin{smallmatrix}1\\0\end{smallmatrix})]$ and the fact that $A$ has an orthonormal basis of eigenvectors, yields that the semigroup is well-defined on $X_{\rm exp}=(\ker A)^{\perp}$ and the norm is bounded by the exponential related to the largest negative eigenvalue of $A$. Finally, since $b\in X_{\rm exp}$, it follows that the range of $B$ lies in $(X_{\rm exp})_{-1}$. }
\end{proof}

Next we show that the inverse Laplace transform of~$\mathbf{H}$ is a measure of bounded total variation on $\R_{\ge0}$, i.e., $\mathbf{H}\in {\rm M}(\R_{\ge0})$, where the total variation of $f\in{\rm M}(\R_{\ge0})$ is denoted by $\|f\|_{{\rm M}(\R_{\ge0})}$.

\begin{Lem}\label{lem:measure}
Let $\sigma_{n}=n\pi \sqrt{h_{0}g}$, $n\in\N$.
The function $\mathbf{H}:\C_{+}\!\to\C$ defined in \eqref{eq:transferfct} can be represented as
\[
    \mathbf{H}(\lambda)=-8h_{0}\sum_{n\in\N}\mathbf{H}_{n}(\lambda)=-8h_{0}
    \sum_{n\in2\N_{0}+1}\frac{\lambda}{\lambda^2+2\mu\lambda+\sigma_{n}^2},
\]
with inverse Laplace transform $\mathfrak{h}=\mathcal{L}^{-1}(\mathbf{H})\in {\rm M}(\R_{\ge0})$. 
 Moreover,
\[
    \mathfrak{h}=\mathfrak{h}_{L^1}+\frac{1}{4c}\mathfrak{h}_{\delta}=\mathfrak{h}_{L^1}+\frac{1}{4c}\delta_{0}+\frac{1}{2c}\sum_{k\in\N} (-1)^{k}e^{-k\frac{\mu}{c}}\delta_{\frac{k}{c}},
\]
where $c=\sqrt{h_{0}g}$, and $\mathfrak{h}_{L^1}(t)=e^{-\mu t}(t^{2}\mathfrak{f}_{2}(t)+t\mathfrak{f}_1(t)+\mathfrak{f}_0(t))$, $t\ge 0$,
for some $\mathfrak{f}_{0},\mathfrak{f}_{1},\mathfrak{f}_{2}\in L^{\infty}(\R_{\ge0};\R)$.
\end{Lem}

\begin{proof}
The \tb{asserted} series representation of $\mathbf{H}$, \tb{with $\mathbf{H}_{n}(\lambda)=\frac{\lambda}{\lambda^2+2\mu\lambda+\sigma_{n}^2}$}, follows from~\eqref{eq:transferfct} and the following \tb{well-known} series representation of the hyperbolic tangent,
$$\tanh(z)=8z\sum_{k=1}^{\infty}\frac{1}{\pi^{2}(2k-1)^{2}+4z^{2}}, \quad z\notin i\pi(1+2\mathbb{Z}).$$
Next we study the inverse Laplace transform of $\mathbf{H}$; in particular, $\mathbf{H}_n(\lambda)=0$ for $n\in 2\N_0$.
It is clear that $\mathbf{H}$ is also continuous on $\overline{\C_{+}}$ and that the series converges locally uniformly along the imaginary axis.
Thus, the partial sums converge to $\alpha\mapsto\mathbf{H}(i\alpha)$ in the distributional sense when considered as tempered distributions on $i\R$. By continuity of the Fourier transform $\mathcal{F}$, this gives that the series
\[
    -8h_0\sum_{n\in\N}\mathcal{F}^{-1}(\mathbf{H}_{n}(i\cdot))=-8h_0\sum_{n\in\N}\mathcal{L}^{-1}(\mathbf{H}_{n})
\]
converges to $\mathfrak{h}=\mathcal{F}^{-1}(\mathbf{H}(i\cdot))=\mathcal{L}^{-1}(\mathbf{H})$  in the distributional sense\footnote{We identify functions on $\R_{\ge 0}$ with their trivial extension to $\R$.}. It remains to study $\mathcal{L}^{-1}(\mathbf{H}_{n})$ and \tb{to show that the limit of the corresponding sum is in  ${\rm M}(\R_{\ge0})$}. By known rules for the Laplace transform, $\mathcal{L}^{-1}(\mathbf{H}_{n})(t)=e^{-\mu t}\mathfrak{g}_{n}(t)$ for $t\ge0$, with $\phi_{n}=\sqrt{\sigma_{n}^{2}-\mu^{2}}$ and
$$\mathfrak{g}_{n}(t)=\cos(\phi_{n} t)-\mu\phi_{n}^{-1}\sin(\phi_{n} t), \quad n\in2\N_0+1.$$
The idea is to use Fourier series that are related to the frequencies $\sigma_{n}$ in contrast to the `perturbed' harmonics $\sin \phi_{n}$ and $\cos \phi_{n}$. We write
\begin{align*}
\mathfrak{g}_n(t)={}&\left[\cos(\phi_{n}t)-\cos(\sigma_{n}t)\right]+\frac{\mu}{\phi_{n}}\left[\sin(\sigma_{n}t)-\sin(\phi_{n}t)\right]\\
&+\cos(\sigma_{n}t)+\frac{\mu}{\phi_{n}}\sin(\sigma_{n}t)
\end{align*}
\tb{and investigate each term in the sum separately.}
By the mean value theorem there exist $\alpha_n,\beta_n\in[\phi_n,\sigma_n]$ and $\omega_{n}\in[\alpha_{n},\sigma_{n}]$ such that
\begin{align*}
\cos(\phi_{n}t)-\cos(\sigma_{n}t)&=t(\sigma_{n}-\phi_{n})\sin(\alpha_{n}t)
=\frac{\mu^2t \sin(\alpha_{n}t)}{\sigma_{n}+\phi_{n}},\\
\sin(\alpha_{n}t)&=t(\alpha_{n}-\sigma_{n})\cos(\omega_{n}t)+\sin(\sigma_{n}t),
\\
\sin(\sigma_{n}t)-\sin(\phi_{n}t)&=t(\sigma_{n}-\phi_{n})\cos(\beta_{n}t)
=\frac{\mu^2t \cos(\beta_{n}t)
}{\sigma_{n}+\phi_{n}},
\end{align*}
\tb{where we used that $\sigma_{n}^{2}-\phi_{n}^{2}=\mu^{2}$}. Hence,
\begin{align*}
\mathfrak{g}_n(t)={}&t^{2}\frac{\mu^2(\alpha_{n}-\sigma_{n})}{\sigma_{n}+\phi_{n}}\cos(\omega_{n}t)+\frac{\mu^3t}{\phi_{n}(\sigma_{n}+\phi_{n})}\cos(\beta_{n}t)\\
&+\cos(\sigma_{n}t)
+\left(t(\sigma_{n}-\phi_{n})+\frac{\mu}{\phi_{n}}\right)\sin(\sigma_{n}t).
\end{align*}
The coefficient sequences of  the first two terms in the sum,
$$a_{n}:=\mu^{2}\frac{\alpha_{n}-\sigma_{n}}{\sigma_{n}+\phi_{n}},\quad b_{n}:=\frac{\mu^{3}}{\phi_{n}(\sigma_{n}+\phi_{n})},$$
are absolutely summable sequences \tb{as $\sigma_{n}+\phi_{n}\lesssim n$ and }
\[
    0 > a_n > \mu^{2}\frac{\phi_{n}-\sigma_{n}}{\sigma_{n}+\phi_{n}} = \frac{-\mu^4}{(\sigma_{n}+\phi_{n})^2}.
\]
Let us  rewrite the coefficient of the last term, recalling that
$\sigma_{n}^{2}-\phi_{n}^{2}=\mu^{2}$ implies that
 $\frac{1}{\sigma_{n}+\phi_{n}}-\frac{1}{2\sigma_{n}}=\frac{\mu^{2}}{2\sigma_{n}(\sigma_{n}+\phi_{n})^{2}}$,
\begin{align*}
t(\sigma_{n}-\phi_{n})={}&\frac{\mu^2t}{\sigma_{n}+\phi_{n}}=\frac{\mu^4t}{2\sigma_{n}(\sigma_{n}+\phi_{n})^2}
+\frac{\mu^2t}{2\sigma_{n}},\\
\frac{\mu}{\phi_{n}}={}&\frac{\mu}{\phi_{n}}+\frac{\mu}{\sigma_{n}}-\frac{\mu}{\sigma_{n}}=
\frac{\mu}{\sigma_{n}}+\frac{\mu^{3}}{\sigma_{n}\phi_{n}(\sigma_{n}+\phi_{n})}.
\end{align*}
Thus, with $c_{n}=\frac{\mu^4}{2\sigma_{n}(\sigma_{n}+\phi_{n})^2}$ and $d_{n}=\frac{\mu^{3}}{\sigma_{n}\phi_{n}(\sigma_{n}+\phi_{n})}$, which define absolutely summable sequences, we have \tb{that}
\begin{align*}
\mathfrak{g}_n(t)&=t^{2}a_{n}\cos(\omega_{n}t)+tb_{n}\cos(\beta_{n}t)+\left(tc_{n}+d_{n}\right)\sin(\sigma_{n}t)\\
&\quad +\cos(\sigma_{n}t)+(\mu t+2)\frac{\mu}{2\sigma_{n}}\sin(\sigma_{n}t).
\end{align*}
\tb{Multiplying with $\mathrm{e}^{-\mu t}$, it is clear that the sums of terms involving $a_{n},b_{n},c_{n},d_{n}$ converge in the $L^{1}$-norm. Thus, it remains to estimate the last two terms in $\mathfrak{g}_{n}$ above. }
As $\sigma_{n}=n\pi c$, the sum $\sum\nolimits_{n\in2\N_{0}+1}\frac{4c}{\sigma_{n}}\sin(\sigma_{n}t)$ converges to
\[
    H_{0}(t)= (-1)^k\quad \text{for } t\in [k/c, (k+1)/c),\ k\in \N_0,
\]
for almost all $t\ge 0$. Therefore, for almost all $t\ge0$,
$$\sum_{n\in2\N_{0}+1}\frac{\mu}{2\sigma_{n}}\sin(\sigma_{n}t)= \frac{\mu}{8c}H_{0}(t).$$
Since the coefficients $\frac{\mu}{\sigma_{n}}$ are square summable, the series even converges in $L^{2}$ on any bounded interval and thus particularly in the distributional sense on $\R_{\ge 0}$.\newline
Finally, by known facts on the Fourier series of delta distributions, $4c\sum_{n\in2\N_{0}+1}\cos(\sigma_{n}\cdot)$ converges to the $\frac{2}{c}$-periodic extension of
$(\delta_{0}-2\delta_{\frac{1}{c}} + \delta_{\frac{2}{c}})$ in the distributional sense as
\begin{align*}
&\lim_{N\to\infty}\left\langle 4c\!\sum_{n=1,\, n\,\text{odd}}^{N}\! \cos(\sigma_{n}\cdot),\psi\right\rangle
=\left\langle \delta_{0}-2\delta_{\frac{1}{c}} + \delta_{\frac{2}{c}},\psi\right\rangle
\end{align*}
 for any function $\psi\in C^{\infty}([0,\frac{2}{c}];\R)$.
Altogether, and as multiplying with $e^{-\mu t}$ preserves the distributional convergence,
this yields that
\begin{align*}
	\sum_{n\in 2\N_{0}+1}\mathcal{L}^{-1}(\mathbf{H}_{n})(\cdot)=\!\sum_{n\in 2\N_{0}+1}e^{-\mu \cdot}\mathfrak{g}_{n}(\cdot)=\mathfrak{h}_{L^1}(\cdot)+\frac{1}{4c}\mathfrak{h}_{\delta}
\end{align*}
with $\mathfrak{h}_{L^1}$, $\mathfrak{h}_{\delta}$ as in the assertion and where the functions
\begin{align*}
\mathfrak{f}_2(t):={}&\sum_{n\in2\N_{0}+1}a_{n}\cos(\omega_{n}t)\\
\mathfrak{f}_1(t):={}&\frac{\mu^{2}}{8c}H_{0}(t)+\sum_{n\in2\N_{0}+1}b_{n}\cos(\beta_{n}t)+c_{n}\sin(\sigma_{n}t),\\
\mathfrak{f}_0(t):={}&\frac{\mu}{4c}H_{0}(t)+\sum_{n\in\N}d_{n}\sin(\sigma_{n}t),\qquad t\ge 0,
\end{align*}
are bounded since  $a_{n},b_{n},c_{n}, d_{n}$ are absolutely summable sequences. By this representation, $\mathfrak{h}_{L^1}\in L^{1}(\R_{\ge0};\R)$ and can thus be identified with an element in ${\rm M}(\R_{\ge 0})$, while $\mathfrak{h}_{\delta}\in {\rm M}(\R_{\ge0})$ as
$\|\mathfrak{h}_{\delta}\|_{{\rm M}(\R_{\ge0})}=1 + 2\sum_{k\in\N}e^{-\mu \frac{k}{c}}<\infty$.
\end{proof}


For regular well-posed systems, it is convenient to consider an extension of the observation operator $C\in\cB(X_1;Y)$, the so-called $\Lambda$-extension $C_{\Lambda}:\cD(C_{\Lambda})\to Y$, defined by
\[ C_{\Lambda}x= \lim_{\lambda\to\infty}\lambda C(\lambda I-A)^{-1}x\]
 with $\cD(C_{\Lambda})=\setdef{x\in X }{ \lim_{\lambda\to\infty}\lambda C(\lambda I-A)^{-1}x\text{ exists}}$. It is easy to see that this indeed defines an extension of $C$, cf.~\cite{TucsWeis14}. In the following we will replace the operator~$C$ from~\eqref{eq:opFTC} by its $\Lambda$-extension; thus, in particular, in~\eqref{eq:opFTC},
 \begin{equation}\label{eq:domC}
    \cD(C) = \cD(C_\Lambda).
 \end{equation}
Recall that the unique strong solution of~\eqref{eq:Teq0} is given by~\eqref{eq:strongsol}.

\begin{Prop}\label{prop:BIBO}
Let $x_{0}\in X$ and $v_0\in\C$ such that $x_0 + b v_0\in\cD(A)$. Then
\[\mathcal{F}:\cC(\R_{\ge 0};\C)\to L_{\mathrm{loc}}^{\infty}(\R_{\ge 0};X\times\C^{2}),\eta\mapsto \begin{smallbmatrix}I\\C_{\Lambda}\end{smallbmatrix}x,\]
with $x$ as in~\eqref{eq:strongsol}, is well-defined and
\begin{equation}\label{eq99}
  \|\mathcal{F}(\eta)\|_{\infty}\lesssim \|x_{0}\|+\|Ax_{0}\|+\|\eta\|_{\infty}
\end{equation}
for all $\eta\in\cC(\R_{\ge 0};\C)\cap L^\infty(\R_{\ge 0};\C)$.
\end{Prop}
\begin{proof}
\tb{Since $(A,B,\begin{smallbmatrix}I\\C\end{smallbmatrix})$ is a regular, well-posed system by Proposition~\ref{prop:sec3} with transfer function $\mathbf{G}$,
it follows that $x(t)\in \cD(\begin{smallbmatrix}I\\C_{\Lambda}\end{smallbmatrix})$ for a.e.\ $t>0$ and that $\mathcal{F}$ is well-defined as a mapping from $L_{\mathrm{loc}}^{2}(\R_{\ge 0};\C)$ to $L_{\mathrm{loc}}^{2}(\R_{\ge 0};X\times\C^{2})$,  see e.g.~\cite[Thm.~5.3]{TucsWeis14} or \cite[Thm.~5.6.5]{Staf05}. We will discuss the two components of the mapping $\mathcal{F}$ separately. By Proposition \ref{prop:sec3}, we have that the semigroup restricted to $X_{\rm exp}$ has a negative growth bound and $B\in\mathcal{B}(\R,(X_{\rm exp})_{-1})$ is admissible, thus it follows from~\cite[Prop.~4.2.4]{TucsWeis09} that, for all $\eta\in\cC(\R_{\ge 0};\C)\cap L^\infty(\R_{\ge 0};\C)$ and all $t\ge 0$, $t\mapsto\int_{0}^{t}T_{-1}(t-s)Bu(s)\mathrm{d}s\in \cC(\R_{\ge0};X)$ and there exists a constant $c$ independent of $t$ and $u$ such that
\[\left\|\int_{0}^{t}T_{-1}(t-s)u(s)\mathrm{d}s\right\|\leq c \|\eta\|_{L^{2}([0,t];\C)}.\]
By rescaling the semigroup we may replace $u(s)$ by $\mathrm{e}^{-\varepsilon(t-s)}u(s)$, for some $\varepsilon>0$, and further estimate the term on the right hand side by $c(2\varepsilon)^{-\frac{1}{2}}\|\eta\|_{\infty}$, invoking Young's convolution inequality. Therefore, $\mathcal{F}_{1}:\eta\mapsto x$ maps $\cC(\R_{\ge 0};\C)$ to $L_{\mathrm{loc}}^{\infty}(\R_{\ge 0};X)$ with $$\|\mathcal{F}_{1}(\eta)\|_{\infty}=\|x(t)\|_{\infty}\lesssim \|x_{0}\|+\|\eta\|_{\infty}.$$
Since $\lim_{s\to\infty,s\in\R}\mathbf{G(s)}=0$, we have for $\mathcal{F}_{2}:\eta\mapsto C_{\Lambda}x$ that
\[
    \mathcal{L}({\mathcal{F}_{2}(\eta)})(s)=C_{\Lambda}(sI-A)^{-1}x_{0}+\mathbf{G}_{2}(s)\cdot\mathcal{L}({\eta})(s)
\]
for all $s\in \C_{+},\eta\in L^{2}(\R_{\ge 0};\C)$ and where $\mathbf{G}_{2}=(\mathbf{H}+2\sqrt{h_{0}/g},0)$ are the functions defined in Proposition~\ref{prop:sec3}. To show that $\mathcal{F}_{2}$ is well-defined from $\mathcal{C}(\R_{\ge 0};\C)$ to  $L_{\mathrm{loc}}^{\infty}(\R_{\ge 0};\C^{2})$, it suffices to consider only bounded, continuous functions $\eta$ and to show that~\eqref{eq99} holds, as the rest follows by causality.  We identify $\mathcal{L}^{-1}(\mathbf{G}_{2})$ and $\eta$ with their trivial extensions to $\R$ and
 get for a.a.\ $t>0$ that
\begin{align} \mathcal{F}_{2}(\eta)(t)={}&C_{\Lambda}T(t)x_{0}+(\mathcal{L}^{-1}(\mathbf{G}_{2})\ast \eta)(t)\notag\\
={}&C_{\Lambda}T(t)(x_{0} + bv_0)+(\mathcal{L}^{-1}(\mathbf{G}_{2})\ast \tilde{\eta})(t), \label{prop33:eq1}
\end{align}
with $\tilde{\eta}(t)=\mathrm{e}^{t}v_{0}\chi_{\R_{\le 0}}(t)+\eta(t)$ and where $\chi_{\R\leq0}$ denotes the indicator function on $\R_{\le0}$. Here, we used the fact that
\begin{align*}
-C_{\Lambda}T(t)b={}&C_{\Lambda}\int_{-\infty}^{0}T_{-1}(t-s)B\mathrm{e}^s\mathrm{d}s\\
	={}&(\mathcal{L}^{-1}(\mathbf{G}_{2})\ast \mathrm{e_{1}}\chi_{\R\leq0})(t)
\end{align*}
for a.a.\ $t>0$, with $\mathrm{e}_{1}(s)=\mathrm{e}^{s}$, $s\in\R$.
The first term on the right hand side of \eqref{prop33:eq1} is uniformly bounded on $\R_{\ge 0}$,
because $x_{0}+bv_{0}\in \cD(A)$, $\|T(t)\|\leq1$ for all $t\ge0$ by Proposition \ref{prop:sec3} and thus
\[\|C_{\Lambda}T(t)(x_{0} + bv_0)\|\leq\|C_{\Lambda}A^{-1}\|_{\mathcal{B}(X;\C^{2})}\|A(x_{0} + bv_0)\|.\]
The uniform boundedness of the second term \tb{in \eqref{prop33:eq1}} follows since $\mathcal{L}^{-1}(\mathbf{G}_{2})$ is of bounded total variation by Lemma~\ref{lem:measure} and thus defines a bounded convolution operator with respect to the supremum norm, cf.~\cite{Graf14}.}
\end{proof}

\section{Extensions}\label{Sec:Ext}

\tb{
In this section we consider some extensions (other steady states, space-dependent friction) of the results of Sections~\ref{Sec:FunCon} and~\ref{Sec:LinMod}. Although the linearization around the steady state $(h_0,0)$ considered in the previous sections is the most relevant from a control theoretic viewpoint, cf. Remark~\ref{Rem:Ext}\,(ii) below, other scenarios may be of interest.
In the following we restrict ourselves to deriving the new model, when~\eqref{eq:SVeq} is linearized around other steady states, and indicate how the approach of Sections~\ref{Sec:FunCon} and~\ref{Sec:LinMod} can be extended.}


\tb{
First we like to note that the steady state with constant water level $h=h_{0}$ and zero velocity $v=0$ is not the only choice for a stationary solution of~\eqref{eq:SVeq}. Indeed, considering as steady states the solutions of the overall system constituted by \eqref{eq:SVeq}, \eqref{eq:momentum} and $\dot{p}(t)=u(t)$ in the variables $h,v,u$, which are constant in time, the steady states are given by all solutions $H, V:[0,1]\to\R$ and $U\in\R$, with $H$ being strictly positive, of
\begin{align*}
	\partial_{\zeta}(HV)={}&0,\\
	\partial_{\zeta}\left(\frac{V^{2}}{2}+gH\right)+HS\left(\frac{V}{H}\right)={}&-\frac{U}{m+h_{0}},
\end{align*}
where the right-hand side of the second equation is derived from taking the time-derivative in \eqref{eq:momentum} and using $h_{0}=\int_{0}^{1}H(\zeta)\mathrm{d}\zeta$. By the boundary conditions for the velocity, we conclude from the first equation that $V=0$. Using that $S(0)=0$, the second equation thus becomes
\[g\partial_{\zeta}H=-\frac{U}{m+h_{0}}.\]
Since $H$ is strictly positive with $h_{0}=\int_{0}^{1}H(\zeta)\mathrm{d}\zeta$, this is solved by the function
\begin{equation}\label{eq:steadyH}
    H(\zeta)= -\frac{U}{m+h_{0}}\zeta+\frac{U}{2(m+h_{0})} + h_0,
\end{equation}
provided that $|U| < 2h_0 (m+h_0)$. The case $U=0$ corresponds to $H\equiv h_0$, which leads to the linearization~\eqref{eq:SVlin}.
}

\vspace{-5mm}

\tb{
\begin{Rem}\label{Rem:Ext}
\begin{enumerate}
  \item In order to derive the steady states for control values $U$ with  $|U| \ge 2h_0 (m+h_0)$ a different approach must be taken, which is only sketched here for completeness. In such cases the height profile~$H$ will no longer be strictly positive in general. Therefore, assuming that~$S$ is bounded by a polynomial of order $\alpha$, i.e., $|S(z)|\le \sum_{i=0}^\alpha c_i |z|^i$, we may multiply the second equation in~\eqref{eq:SVeq} on both sides with $h^\alpha$, leading to a partial differential-algebraic equation (PDAE). Computing the steady states of this equation again leads to $V=0$ and~$H$ is determined by
      \[H^\alpha g\partial_{\zeta}H=-\frac{H^\alpha U}{m+h_{0}},\]
      which has, under the additional condition that $h_{0}=\int_{0}^{1}H(\zeta)\mathrm{d}\zeta$, unique weak solutions that may be zero on a subinterval of $[0,1]$, and linear otherwise~-- we leave the exact computation to the reader. Let $H^*$ be such a solution, then the steady states are $(H^*, 0, U)$, where $U\in\R$ is no longer restricted. The left-hand side of the linearization of the original PDAE around this steady state then reads $\begin{smallpmatrix} \partial_t z_1 \\ (H^*)^\alpha\partial_t z_2\end{smallpmatrix}$, and the second component vanishes whenever $H^*$ is zero. We may further compute that this is also true for the second component of the right-hand side, so only the first equation, reading $\partial_t z_1 = 0$ is present whenever $H^* = 0$. Therefore, the linearization is also a PDAE and cannot be described in a simple form as in~\eqref{eq:SVlin}.
  \item On the other hand, the linearization around steady states with non-zero control values may be of limited interest from a control theoretic viewpoint, since typically situations are considered where the system is steered from one operating point to another, i.e., the control input has compact support. This rather suggests to consider equilibria with zero control value $U=0$ in order to conclude that the controlled linearized system approximates the controlled nonlinear system in a certain sense.
\end{enumerate}
\end{Rem}
}

\tb{
Linearizing around the steady state $(H,0,U)$ with~$H$ as in~\eqref{eq:steadyH} gives the following generalization of \eqref{eq:SVlin},
\begin{equation}\label{eq:SVlin2}
    \partial_t z = -\begin{bmatrix}0&H\partial_{\zeta}+(\partial_{\zeta}H) \\g\partial_\zeta &2\mu\end{bmatrix}z+\begin{pmatrix}0\\-1\end{pmatrix}\ddot{y}
\end{equation}
with the same boundary conditions as in Section~\ref{Ssec:MathModel}. 
The state space in which $z(t)$ evolves is
\[
 \setdef{(z_1,z_2)}{ z_1, (Hz_2) \in L^2([0,1];\R)} = L^2([0,1];\R^2) = X
\]
and the new operator, parameterized by the steady state control value~$U$, is $A_U: \cD(A_U) \subseteq  X \to X$,
\begin{align*}
A_U z &= -\begin{bmatrix}0&H\partial_{\zeta}+(\partial_{\zeta}H) \\g\partial_\zeta &2\mu\end{bmatrix}z,\\
    \cD(A_U) &= \setdef{(z_1,z_2)\in X}{ \!\!\begin{array}{l} z_1,(Hz_2)\in W^{1,2}([0,1];\R),\\ z_2(0) = z_2(1) = 0\end{array}\!\!\!}.
\end{align*}
Additionally we allow for a space-dependent friction term $\mu:[0,1]\to \R_{>0}$ in the following. Then, again invoking the momentum~\eqref{eq:momentum} and observing that by~\eqref{eq:steadyH} and $z_2(t,0) = z_2(t,1) = 0$ we have
\begin{align*}
    &\int_0^1 \!\! \partial_\zeta\big(H(\zeta) z_2(t,\zeta)\big) z_2(t,\zeta)\! \ds{\zeta} \!=\! - \frac12  \int_0^1 \!\! H(\zeta) \partial_\zeta\big( z_2(t,\zeta)^2 \big)\! \ds{\zeta} \\
    & =\! - \frac{U}{2(m+h_0)}  \int_0^1 \zeta \partial_\zeta\big( z_2(t,\zeta)^2 \big) \ds{\zeta} \!=\! \frac{U}{2(m+h_0)} \|z_2(t)\|^2,
\end{align*}
the nonlinear model on the state space $X$ reads
\begin{subequations}\label{eq:InpOut_U}
\begin{align}\label{eq:InpOutlin_U}
	  \partial_t x &= A_U (x+b\dot{y}) \\
 \!  m \ddot y(t) &= \frac{g}{2}  x_1(t,\cdot)^{2}|_{0}^{1} + {2}\langle x_{1}(t),\mu x_{2}(t)\rangle-2 \dot{y}(t) \langle x_{1}(t),\mu \rangle\notag\\
 &\quad + \frac{U}{2(m+h_0)} \|x_2(t) - \dot y(t)\|^2 +{u(t)} \label{eq:InpOutlin2_U}
\end{align}
\end{subequations}
}
\tb{
This generalized system can be approached in a similar way as in Section~\ref{Sec:LinMod}, however, although the steady states are explicitly given by~\eqref{eq:steadyH}, the computations for the transfer function, in particular the crucial Lemma~\ref{lem:measure}, have to be adapted accordingly. These technicalities will not presented here as the authors believe that a more abstract approach for the assessment of {\it bounded-input, bounded-output stability} for linear systems of port-Hamiltonian form, see e.g.~\cite{JacoZwar12}, should be considered. More precisely, it is not hard to show that the linear dynamics in \eqref{eq:InpOut_U}, linking $\dot{y}(t)$ to the spatial boundary values of $x(t)$, can be rewritten in the port-Hamiltonian form
\begin{align*}
\partial_t {x}={}&(P_{1}\partial_{\zeta}-\mathcal{R})\mathcal{H}x\\
\dot{y}(t)={}&W_{B}\begin{bmatrix}f_{\partial}(t)\\e_{\partial}(t)\end{bmatrix},\quad
x(\cdot,t)|_{0,1}=W_{C}\begin{bmatrix}f_{\partial}(t)\\e_{\partial}(t)\end{bmatrix},
\end{align*}
for suitable real matrices $W_{B},W_{C}$, where
$$P_{1}=-\begin{bmatrix}0&1 \\1 &0\end{bmatrix},\ \mathcal{R}=\begin{bmatrix}0&0\\0&2\mu H^{-1}\end{bmatrix}, \ \mathcal{H}=\begin{bmatrix}g^{-1}&0\\0&H\end{bmatrix}$$ and
$$ \begin{bmatrix}f_{\partial}(t)\\e_{\partial}(t)\end{bmatrix}=\frac{1}{\sqrt{2}}\begin{bmatrix}P_{1}&-P_{1}\\I&I\end{bmatrix}\begin{bmatrix} \mathcal{H}(0)x(0,t)\\\mathcal{H}(1)x(1,t)\end{bmatrix}.$$
The abstract characterization of when the corresponding transfer function has a form as in Lemma~\ref{lem:measure} is a subject of future work.
}

\section{Simulations}\label{Sec:Sim}
%

In this section we illustrate the application of the funnel controller~\eqref{eq:fun-con} to the system~\eqref{eq:InpOut}.
 Using the change of variables $z(t,\zeta)=Q\begin{smallpmatrix} \eta_1(t,\zeta)\\ \eta_2(t,\zeta)\end{smallpmatrix}$ with $Q:=\begin{smallbmatrix}1 & 1\\ \frac{g}{c} & -\frac{g}{c} \end{smallbmatrix}$ in~\eqref{eq:SVlin} enables us to solve the PDE corresponding to~$\eta_1$ with an implicit finite difference method and the PDE corresponding to~$\eta_2$ with an explicit finite difference method, respectively.
For the simulation we have used the parameters $m=1{\rm kg}$, $h_0=0.5{\rm m}$, $g=9.81{\rm ms^{-2}}$, $\mu=0.1{\rm Hz}$ and the reference signal $y_{\rm ref}(t)=\tanh^2(\omega t)$ with $\omega=0.06\pi f$, $f=\sqrt{g/h_0}$. 
The initial values~\eqref{eq:IClin} are chosen as $x_0(\zeta)=(h_0,0.1\sin^2(4\pi\zeta){\rm ms^{-1}})$
and $(y^0,y^1)=(0{\rm m},0{\rm ms^{-1}})$. For the controller~\eqref{eq:fun-con} we chose the funnel functions $\varphi_0(t)=\varphi_1(t)=100 \tanh(\omega t)$. Clearly, Condition~\eqref{eq:initialerror} is satisfied.
For the finite differences we used a grid in $t$ with $M=2000$ points for the interval $[0,2\tau]$ with $\tau=f^{-1}$, and a grid in $\zeta$ with $N=\lfloor ML/(4c\tau)\rfloor$ points. The method has been implemented in Python and the simulation results are shown in Figs.~\ref{fig:sim-y} and~\ref{fig:sim-e}.
\begin{figure}[h!tb]
\centering
  \includegraphics[trim=0 30 0 0,clip,width=7cm]{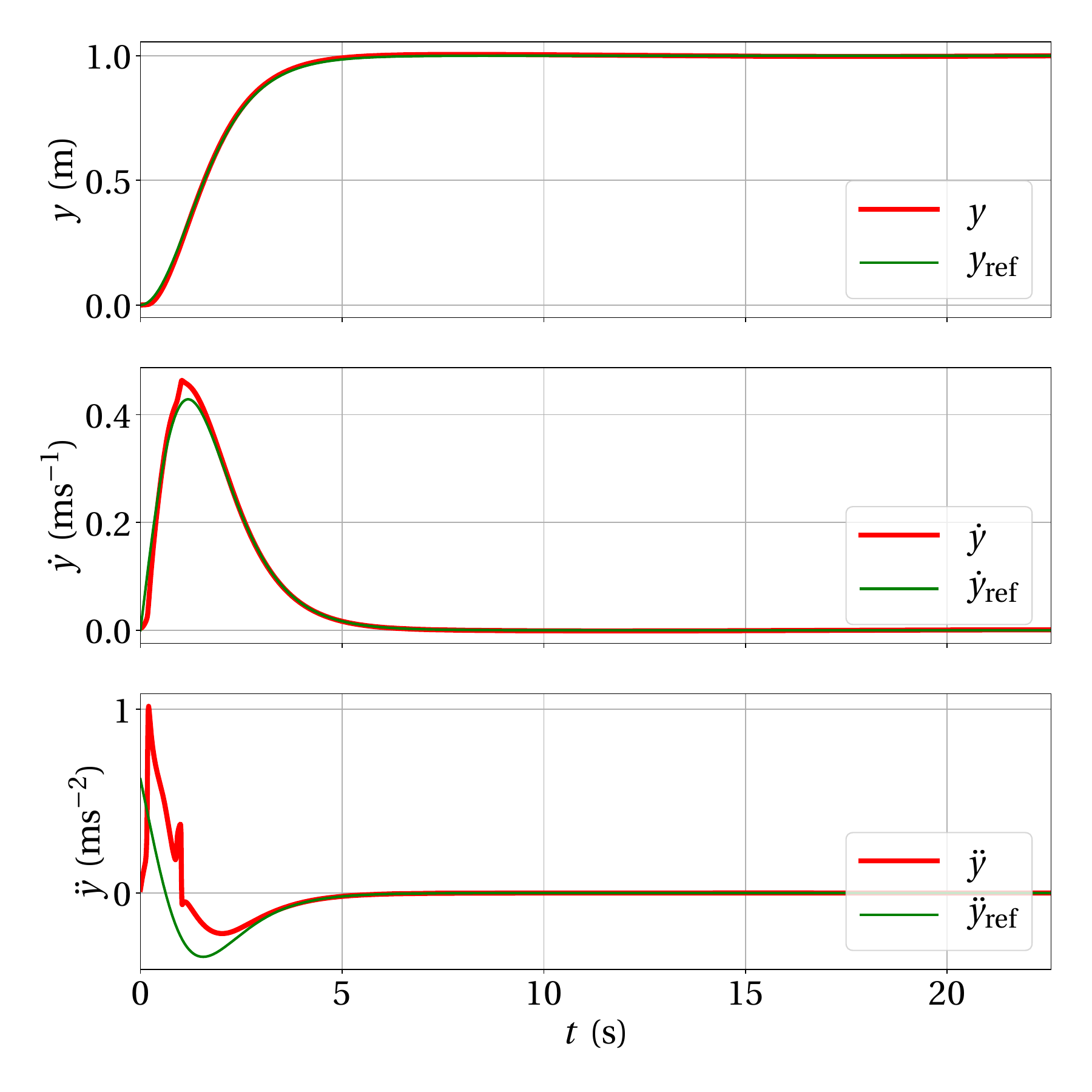}
\caption{\small Output~$y$, reference signal~$y_{\rm ref}$ and corresponding first and second derivatives.}
\label{fig:sim-y}
\end{figure}
\begin{figure}[h!tb]
\centering
  \includegraphics[trim=0 30 0 0,clip,width=7cm]{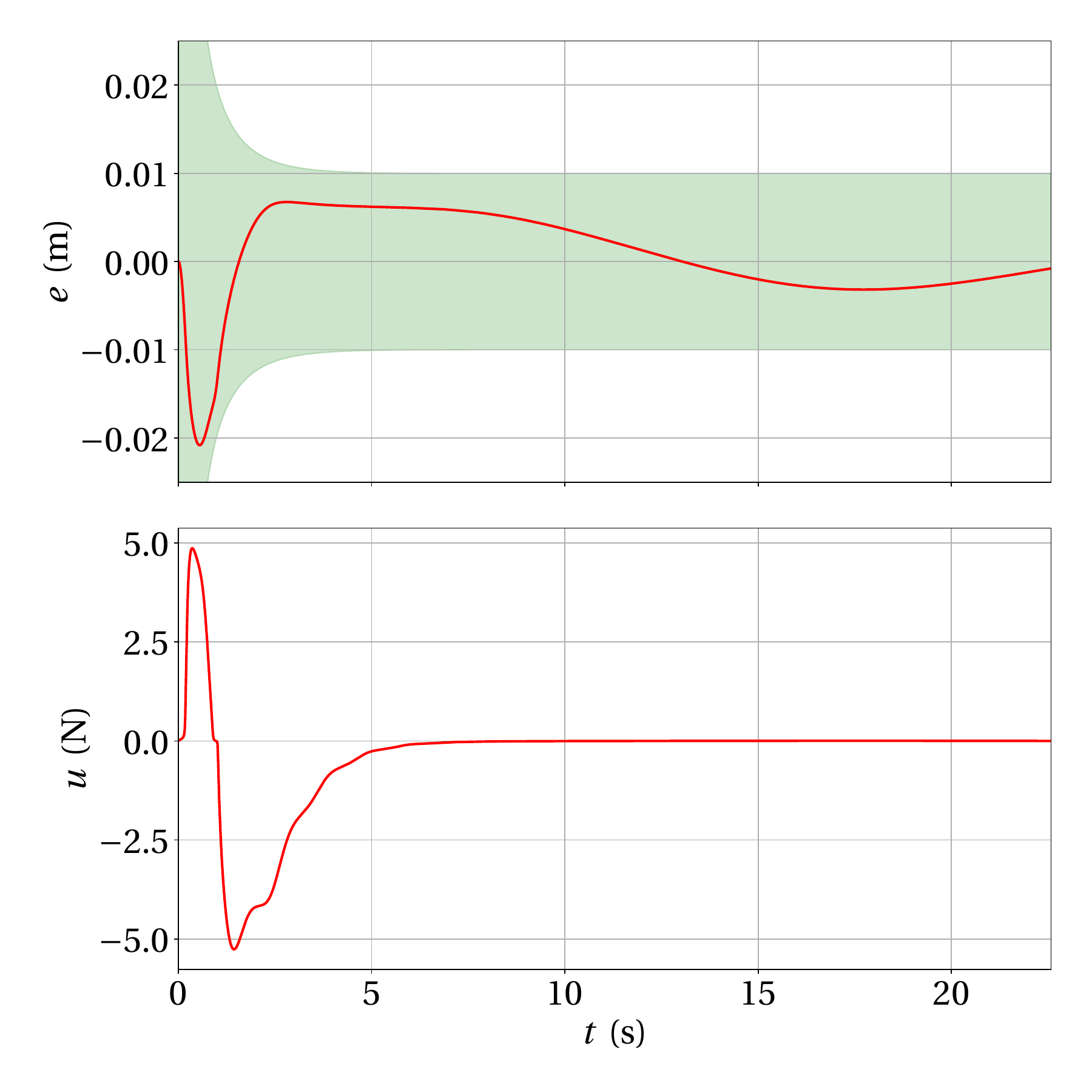}
\caption{\small Performance funnel with tracking error~$e$ and input~$u$.}
\label{fig:sim-e}
\end{figure}

It can be seen that even in the presence of sloshing effects a prescribed performance of the tracking error can be achieved with the funnel controller~\eqref{eq:fun-con}, while at the same time the generated input is bounded and exhibits an acceptable performance.
Finally, we demonstrate that the controller~\eqref{eq:fun-con} is also feasible for the nonlinear Saint-Venant equations~\eqref{eq:SVeq} in certain situations. For purposes of illustration we consider a friction term of the form $S(z) = C_D z + C_S z^2$ with $C_D = 2\mu$ and $C_S = 1$. Analogously as for the linearized model, utilizing the momentum leads to the equation
\begin{equation}\label{eq:InpOutNonlin}
\begin{aligned}
  m \ddot y(t) &= u(t) + C_D \int_0^1 h(t,\zeta)v(t,\zeta) \ds{\zeta} \\
  &\quad + C_S\int_0^1 v(t,\zeta)^2 \ds{\zeta} +  \frac{g}{2}\big( h(t,1)^2-h(t,0)^2\big),
\end{aligned}
\end{equation}
which is used for the simulation instead of~\eqref{eq:InpOutlin2} by applying the same method as described above. We choose the same parameters as in the first simulation, except for $\mu=0.01\, \mathrm{Hz}$, $\omega=0.025\, \mathrm{Hz}$ and $\varphi_0(t)=\varphi_1(t)=10 \tanh(\omega t)$. We compare simulations of the nonlinear Saint-Venant equations~\eqref{eq:SVeq} with initial values $\big(h(0,\zeta),v(0,\zeta)\big)=(h_0,0{\rm ms^{-1}})$ under control~\eqref{eq:fun-con} with the linearized  equations~\eqref{eq:InpOutlin} with initial values $x_0(\zeta)=(h_0,0{\rm ms^{-1}})$ under control~\eqref{eq:fun-con};
the results are shown in Figs.~\ref{fig:nl-sim-y} and~\ref{fig:nl-sim-e}. Note that compared to the first simulation, lower frequencies~$\mu$ and~$\omega$ and a zero initial velocity are chosen here.
}


\begin{figure}[h!tb]
\centering
  \includegraphics[trim=0 30 50 70,clip,width=7cm]{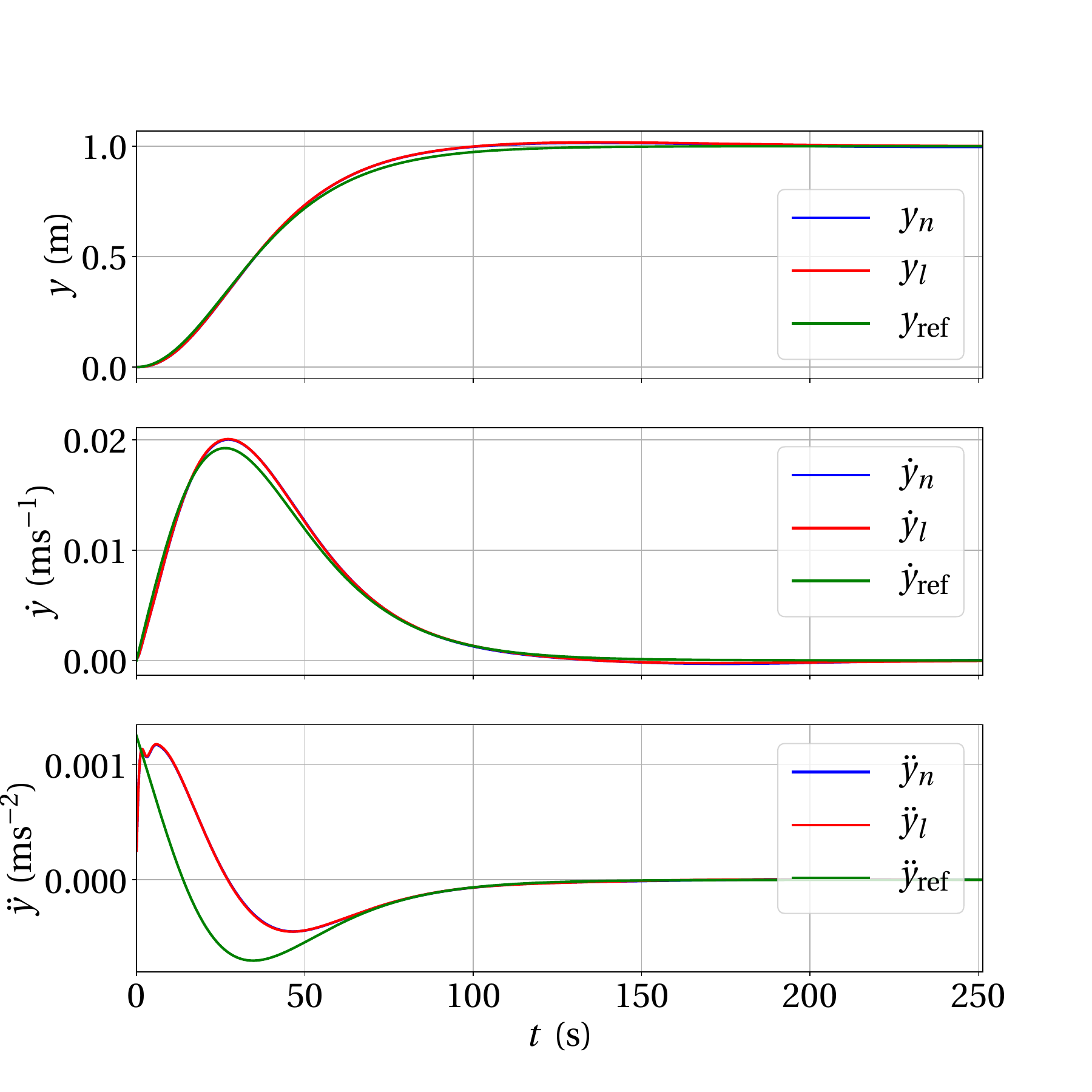}
\caption{\small Reference signal~$y_{\rm ref}$, outputs~$y_n$ for~\eqref{eq:SVeq},~\eqref{eq:fun-con} and~$y_l$ for~\eqref{eq:SVlin},~\eqref{eq:fun-con} and corresponding first and second derivatives.}
\label{fig:nl-sim-y}
\end{figure}


\begin{figure}[h!tb]
\centering
  \includegraphics[trim=0 30 50 70,clip,width=7cm]{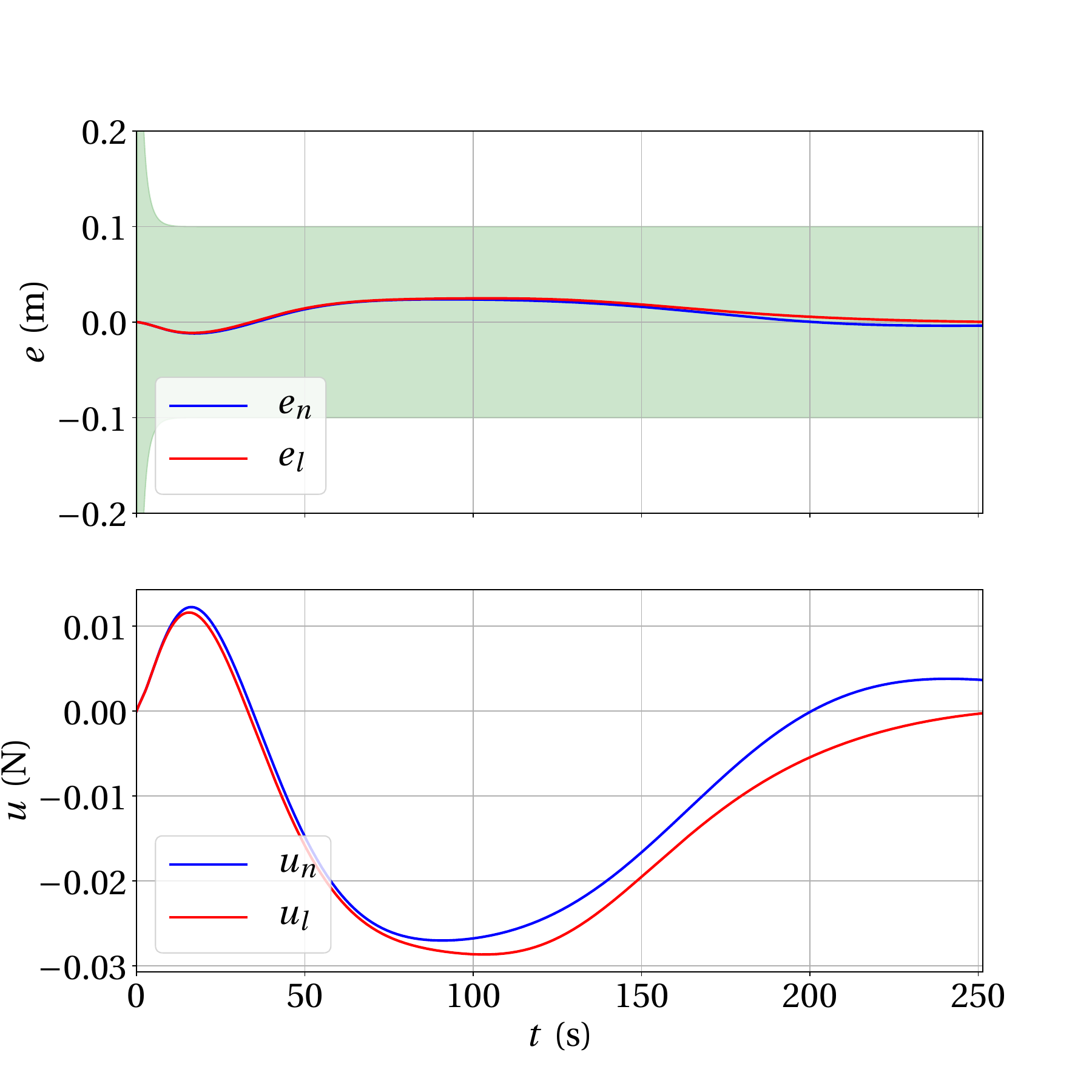}
\caption{\small Performance funnel with tracking errors~$e_n$ and~$e_l$ and inputs~$u_n$ and~$u_l$ for~\eqref{eq:SVeq},~\eqref{eq:fun-con} and~\eqref{eq:SVlin},~\eqref{eq:fun-con}, respectively.}
\label{fig:nl-sim-e}
\end{figure}
%
%

\section{Conclusion}\label{Sec:Concl}

In the present paper we have shown that the controller~\eqref{eq:fun-con} is feasible for the moving water tank system~\eqref{eq:InpOut} which rests on the linearized Saint-Venant equations. We stress that the system is still nonlinear as the (linearized) sloshing effects influence the position of the cart through the momentum, which constitutes an internal feedback loop in~\eqref{eq:InpOut}.  
Nevertheless, the funnel controller is able to handle these effects as shown in Theorem~\ref{Thm:funcon-PDE} and in the simulations in Section~\ref{Sec:Sim}.

We stress that the applicability of the results from~\cite{BergLe18a,BergPuch20a} on funnel control strongly rests  on the fact that the original open-loop system can be viewed as an ODE-PDE coupling with an input-output relation allowing for a relative degree, i.e., the form~\eqref{eq:nonlSys} mentioned in Section~\ref{Ssec:ContrObj}. This, however, is in general not the case for systems governed by evolution equations and different approaches are required then, see~\cite{BergBrei21,PuchReis19pp,ReisSeli15b}.
Furthermore, we like to point out that the controller~\eqref{eq:fun-con} requires that the derivative of the output is available for control. This may not be true in practice, and it may even be hard to obtain suitable estimates of the output derivative. This drawback may be resolved by combining the controller~\eqref{eq:fun-con} with a funnel pre-compensator as developed in~\cite{BergReis18a, BergReis18b}, which results in a pure output feedback.

\tb{Some extensions of the results, such as linearizations around other steady states and space-dependent friction, have been discussed in Section~\ref{Sec:Ext}, but a complete study is subject of future work.} Other extensions of~\eqref{eq:InpOut} which may be considered in future research are e.g.\ sloshing suppressing valves inside the tank, the interconnection of the tank with a truck as in~\cite{GerdKimm15} and, of course, the general nonlinear equations~\eqref{eq:SVeq} as well as the higher-dimensional case.

Another issue is that we assume $\mu >0$ for the friction term. This implies that the system's energy converges to the steady state exponentially. In the case $\mu=0$ the statement of Theorem~\ref{Thm:funcon-PDE} is false in general. More precisely, if $\mu=0$, then $\mathfrak{h} = \cL^{-1}(\mathbf{H})$ derived in Lemma~\ref{lem:measure} does not have bounded total variation, by which $\widetilde{\mathcal{T}}$ from the proof of Theorem~\ref{Thm:funcon-PDE} is not bounded-input, bounded-output stable. \tb{This is consistent with the results from~\cite{DuboPeti99}, where it is shown that the linearized Saint-Venant equations (without damping) are not stabilizable. Nevertheless, as suggested by the findings in~\cite{PrieHall04}, the nonlinear model consisting of~\eqref{eq:SVeq} together with~\eqref{eq:InpOutNonlin} may still have a solution under the control~\eqref{eq:fun-con} in the case $S=0$.}

It seems natural to assume some kind of damping in~\eqref{eq:SVeq}, but one may relax the assumption of exponential stability, even in the linearized case. This requires refined methods, whose development is a topic of future research. In this context, let us mention the recent work~\cite{PeiTucs20}, where stabilization of a linearized 2D-shallow water system subject to polynomial damping was considered.

\section*{Acknowledgements}
We thank T.~Reis (U Hamburg) for fruitful discussions.

\bibliographystyle{elsarticle-harv}

\end{document}